\documentclass[12pt]{article}
\usepackage{amsmath,amssymb,amsthm,amsfonts,latexsym,amscd}
\newtheorem{theorem}{Theorem}[section]
\newtheorem{corollary}[theorem]{Corollary}
\newtheorem{lemma}[theorem]{Lemma}
\newtheorem{proposition}[theorem]{Proposition}

\begin{document}
\title{A characterization of the family of secant lines to a hyperbolic quadric in $PG(3,q)$, $q$ odd}
\author{Puspendu Pradhan \and Bikramaditya Sahu}
\maketitle

\begin{abstract}
We give a combinatorial characterization of the family of lines of $PG(3,q)$ which meet a hyperbolic quadric in two points (the so called secant lines) using their intersection properties with the points and planes of $PG(3,q)$.\\


\end{abstract}

\section{Introduction}
Throughout, $q$ is a prime power. Let  $PG(3,q)$ denote the three-dimensional Desarguesian projective space defined over a finite field of order $q$. Characterizations of the family of external or secant lines to an ovoid/quadric in $PG(3,q)$ with respect to certain combinatorial properties have been given by several authors. A characterization of the family of secant lines to an ovoid in $PG(3,q)$ was obtained in \cite{Fer-Tal} for $q$ odd and in \cite{de-Res} for $q>2$ even, which was further improved in \cite{DO} for all $q>2$. A characterization of the family of external lines to a hyperbolic quadric in $PG(3,q)$ was given in \cite{DGDO} for all $q$ (also see \cite{IZZ} for a different characterization in terms of a point-subset of the Klein quadric in $PG(5,q)$) and to an ovoid in $PG(3,q)$ was obtained in \cite{DO} for all $q>2$. One can refer to \cite{Bar-But8,Bar-But9, Zan,Zua} for characterizations of external lines in $PG(3,q)$ with respect to quadric cone, oval cone and hyperoval cone. Here we give a characterization of the secant lines with respect to a hyperbolic quadric in $PG(3,q)$, $q$ odd.

Let $\mathcal{Q}$ be a hyperbolic quadric in $PG(3,q)$, that is, a non-degenerate quadric of Witt index two. One can refer to \cite{Hir} for the basic properties of the points, lines and planes of $PG(3,q)$ with respect to $\mathcal{Q}$. Every line of $PG(3,q)$ meets $\mathcal{Q}$ in $0, 1, 2$ or $q+1$ points. A line of $PG(3,q)$ is called {\it secant} with respect to $\mathcal{Q}$ if it meets $\mathcal{Q}$ in 2 points. The lines of $PG(3,q)$ that meet $\mathcal{Q}$ in $q+1$ points are called {\it generators} of $\mathcal{Q}$. Each point of $\mathcal{Q}$ lies on two generators. The quadric $\mathcal{Q}$ consists of $(q+1)^2$ points and $2(q+1)$ generators.

The total number of secant lines of $PG(3,q)$ with respect to $\mathcal{Q}$ is $q^2(q+1)^2/2$. We recall the distribution of secant lines with respect to points and planes of $PG(3,q)$ which plays an important role in this paper. Each point of $\mathcal{Q}$ lies on $q^2$ secant lines. Each point of $PG(3,q)\setminus\mathcal{Q}$ lies on $q(q+1)/2$ secant lines.  Each plane of $PG(3,q)$ contains $q^2$ or $q(q+1)/2$ secant lines. If a plane contains $q^2$ secant lines, then every pencil of lines in that plane contains $0$ or $q$ secant lines. If a plane contains $q(q+1)/2$ secant lines, then every pencil of lines in that plane contains $(q-1)/2, (q+1)/2$ or $q$ secant lines.\medskip

In this paper, we prove the following theorem when $q$ is odd.

\begin{theorem} \label{q-odd}
Let $\mathcal{S}$ be a family of lines of $PG(3,q)$, $q$ odd, for which the following properties are satisfied:
\begin{enumerate}
\item [(P1)] There are $q(q+1)/2$ or $q^2$ lines of $\mathcal{S}$ through a given point of $PG(3,q)$. Further, there exists a point which is contained in $q(q+1)/2$ lines of $\mathcal{S}$ and a point which is contained in $q^2$ lines of $\mathcal{S}$.
\item [(P2)] Every plane $\pi$ of $PG(3,q)$ contains $q(q+1)/2$ or $q^2$ lines of $\mathcal{S}$. Further,
\begin{enumerate}
\item[(P2a)] if $\pi$ contains $q^2$ lines of $\mathcal{S}$, then every pencil of lines in $\pi$ contains $0$ or $q$ lines of $\mathcal{S}$.
\item[(P2b)] if $\pi$ contains $q(q+1)/2$ lines of $\mathcal{S}$, then every pencil of lines in $\pi$ contains $(q-1)/2, (q+1)/2$ or $q$ lines of $\mathcal{S}$.
\end{enumerate}
\end{enumerate}
Then either $\mathcal{S}$ is the set of all secant lines with respect to a hyperbolic quadric in $PG(3,q)$, or the set of points each of which is contained in $q^2$ lines of $\mathcal{S}$ form a line $l$ of $PG(3,q)$ and $\mathcal{S}$ is a hypothetical family of $\dfrac{q^4+q^3+2q^2}{2}$ lines of $PG(3,q)$ not containing $l$.
\end{theorem}

\section{Combinatorial results}

Let $\mathcal{S}$ be a set of lines of $PG(3,q)$ for which the properties (P1), (P2), (P2a) and (P2b) stated in Theorem \ref{q-odd} hold. A plane of $PG(3,q)$ is said to be {\it tangent} or {\it secant} according as it contains $q^2$ or $\frac{q(q+1)}{2}$ lines of $\mathcal{S}$. For a given plane $\pi$ of $PG(3,q)$, we denote by $\mathcal{S}_\pi$ the set of lines of $\mathcal{S}$ which are contained in $\pi$. By property (P2), $\left|\mathcal{S}_\pi\right|=q^2$ or $\frac{q(q+1)}{2}$ according as $\pi$ is a tangent plane or not. 

We first show that both tangent and secant planes exist. We call a point of $PG(3,q)$ {\it black} if it is contained in $q^2$ lines of $\mathcal{S}$.

\begin{lemma}\label{lem-plane-point}
Let $l$ be a line of $PG(3,q)$. Then the number of tangent planes through $l$ is equal to the number of black points contained in $l$.
\end{lemma}

\begin{proof}
Let $t$ and $b$, respectively, denote the number of tangent planes through $l$ and the number of black points contained in $l$.  We count in two different ways the total number of lines of $\mathcal{S}\setminus\{l\}$ meeting $l$. Any line of $\mathcal{S}$ meeting $l$ is contained in some plane through $l$. If $l\in\mathcal{S}$, then we get
$$t(q^2-1)+(q+1-t)\left(\frac{q(q+1)}{2}-1\right)=b(q^2-1)+(q+1-b)\left(\frac{q(q+1)}{2}-1\right).$$
If $l\notin \mathcal{S}$, then we get
$$tq^2+(q+1-t)\frac{q(q+1)}{2}=bq^2+(q+1-b)\frac{q(q+1)}{2}.$$
In both cases, it follows that $(t-b)\frac{q^2 -q}{2}=0$ and hence $t=b$.
\end{proof}

\begin{corollary}
Both tangent and secant planes exist.
\end{corollary}

\begin{proof}
By property (P1), let $x$ (respectively, $y$) be a point of $PG(3,q)$ which is contained in $q^2$ (respectively, $\frac{q(q+1)}{2}$) lines of $\mathcal{S}$. Taking $l$ to be the line through $x$ and $y$, the corollary follows from Lemma \ref{lem-plane-point} using the facts that $x$ is a black point but $y$ is not a black point.
\end{proof}

As a consequence of Lemma \ref{lem-plane-point}, we have the following.

\begin{corollary}\label{cor-plane}
Every line of a tangent plane contains at least one black point.
\end{corollary}

\begin{corollary}\label{cor-plane-1}
Every black point is contained in some tangent plane.
\end{corollary}

\subsection{Tangent planes}

Note that, by property (P2a), each point of a tangent plane $\pi$ is contained in no line or $q$ lines of $\mathcal{S}_{\pi}$.

\begin{lemma}\label{size-a_pi}
Let $\pi$ be a tangent plane. Then there are $q^2+q$ points of $\pi$, each of which is contained in $q$ lines of $\mathcal{S}_{\pi}$. Equivalently, there is only one point of $\pi$ which is contained in no line of $\mathcal{S}_{\pi}$.
\end{lemma}

\begin{proof}
Let $A_\pi$ (respectively, $B_\pi$) be the set of points of $\pi$ each of which is contained in no line (respectively, $q$ lines) of $\mathcal{S}_\pi$. Then $|A_\pi| +|B_\pi| =q^2+q+1$. We show that $|A_\pi| =1$ and $|B_\pi| =q^2+q$.

Observe that if $l$ is a line of $\mathcal{S}_\pi$, then each of the $q+1$ points of $l$ lies on $q$ lines of $\mathcal{S}_\pi$ by property (P2a) and hence is contained in $B_{\pi}$. Consider the following set of point-line pairs:
$$X=\{(x,l):x \in B_\pi , l\in \mathcal{S}_\pi, x\in l \}.$$
Counting $|X|$ in two ways, we get $|B_\pi|\times q=|X|=q^2\times (q+1)$. This gives $|B_\pi|=q^2+q$ and hence $|A_\pi|=1$.
\end{proof}

For a tangent plane $\pi$, there is a unique point of $\pi$ which is contained in no line of $\mathcal{S}_{\pi}$ by Lemma \ref{size-a_pi}. We denote this unique point by $p_\pi$  and call it the {\it pole} of $\pi$.

\begin{corollary}\label{cor-pole}
Let $\pi$ be a tangent plane. Then the $q+1$ lines of $\pi$ not contained in $\mathcal{S}_\pi$ are precisely the lines of $\pi$ through the pole $p_\pi$.
\end{corollary}

\subsection{Secant planes}

By property (P2b), each point of a secant plane $\pi$ is contained in $\frac{q-1}{2}, \frac{q+1}{2}$ or $q$ lines of $\mathcal{S}_{\pi}$. For a secant plane $\pi$, we denote by $\alpha(\pi)$, $\beta(\pi)$ and $\gamma(\pi)$ the set of points of $\pi$ which are contained in $\frac{q-1}{2}$, $\frac{q+1}{2}$ and $q$ lines of $\mathcal{S}_\pi$, respectively. Similarly, for a line $l$ of a secant plane $\pi$, we denote by $\alpha(l)$, $\beta(l)$ and $\gamma(l)$ the set of points of $l$ which are contained in $\frac{q-1}{2}, \frac{q+1}{2}$ and $q$ lines of $\mathcal{S}_\pi$, respectively. We have $\alpha(l)=l\cap \alpha(\pi)$, $\beta(l)= l\cap \beta(\pi)$ and $\gamma(l)=l\cap\gamma(\pi)$.

\begin{lemma}\label{points-l}
Let $\pi$ be a secant plane and $l$ be a line of $\pi$. Then the following hold:
\begin{enumerate}
\item [(i)] If $l\in \mathcal{S}_\pi$, then $(|\alpha (l)|,|\beta (l)|,|\gamma (l)|)=\left(\frac{q-1}{2},\frac{q-1}{2},2\right)$ or $(0,q,1)$.
\item [(ii)] If $l\notin \mathcal{S}_\pi$, then $(|\alpha (l)|,|\beta (l)|,|\gamma (l)|)=\left(\frac{q+1}{2},\frac{q+1}{2},0\right)$ or $(q,0,1)$.
\end{enumerate}
\end{lemma}

\begin{proof}
We have $|\alpha(l)|+|\beta(l)|+|\gamma(l)|=q+1$, that is, $|\gamma(l)|=q+1- |\alpha(l)|-|\beta(l)|$.
We first assume that $l \in \mathcal{S}_\pi$. Counting the total number of lines of $\mathcal{S}_\pi\setminus\{l\}$ meeting $l$, we get
$$|\alpha(l)|\left(\frac{q-1}{2}-1\right)+|\beta(l)|\left(\frac{q+1}{2}-1\right)+|\gamma(l)|(q-1)=
|\mathcal{S}_\pi|-1=\frac{q(q+1)}{2}-1.$$
This gives
\begin{equation}\label{eq-1}
|\alpha(l)|\left(\frac{q-3}{2}\right)+|\beta(l)|\left(\frac{q-1}{2}\right)+|\gamma(l)|(q-1)=\frac{q^2+q-2}{2}.
\end{equation}
Putting the value of $|\gamma(l)|$ in equation (\ref{eq-1}), we thus get
$$|\alpha(l)|\left(\frac{q+1}{2}\right)+|\beta(l)|\left(\frac{q-1}{2}\right)=\frac{q(q-1)}{2},$$
that is,
$$|\alpha(l)|\left(\frac{q+1}{2}\right)=\left(q-|\beta(l)|\right)\left(\frac{q-1}{2}\right).$$
Since $\frac{q+1}{2}$ and $\frac{q-1}{2}$ are co-prime (being consecutive integers), it follows that $\frac{q+1}{2}$ must divide $q-|\beta(l)|$ and so
$$|\beta(l)| \equiv q \mod \frac{q+1}{2}.$$
Since $0\leq |\beta(l)|\leq q+1$, we have $|\beta(l)|=\frac{q-1}{2}$ or $q$. Consequently, $(|\alpha(l)|,|\beta(l)|,|\gamma(l)|)=\left(\frac{q-1}{2},\frac{q-1}{2},2\right)$ or $(0,q,1)$. This proves (i).

Now assume that $l\notin \mathcal{S}_\pi$. If $|\beta(l)|=q+1$, then the numbers lines of $\mathcal{S}$ which are contained in $\pi$ would be $(q+1)\left(\frac{q+1}{2}\right)$ which is greater than $|\mathcal{S}_\pi | =\frac{q(q+1)}{2}$, a contradiction. So $0\leq |\beta(l)|\leq q$. Counting the total number of lines of $\mathcal{S}_\pi\setminus\{l\}$ meeting $l$, we get
\begin{equation}\label{eq-3}
|\alpha(l)|\left(\frac{q-1}{2}\right)+|\beta(l)|\left(\frac{q+1}{2}\right)+|\gamma(l)|q=|\mathcal{S}_\pi|=\frac{q(q+1)}{2}.
\end{equation}
Putting the value of $|\gamma(l)|$ in equation (\ref{eq-3}), we get
$$|\alpha(l)|\left(\frac{q+1}{2}\right)+|\beta(l)|\left(\frac{q-1}{2}\right)=\frac{q(q+1)}{2},$$
that is,
$$|\beta(l)|\left(\frac{q-1}{2}\right)=(q-|\alpha(l)|)\left(\frac{q+1}{2}\right).$$
If $|\alpha(l)|=q$, then $|\beta(l)|=0$ and so $(|\alpha (l)|,|\beta (l)|,|\gamma (l)|)=(q,0,1)$. Suppose that $|\alpha(l)|\neq q$.
Since the integers $\frac{q+1}{2}$ and $\frac{q-1}{2}$ are co-prime, it follows that $\frac{q+1}{2}$ divides $|\beta(l)|$. Then the restriction on $|\beta(l)|$ that $0\leq |\beta(l)|\leq q$ implies $|\beta(l)|=0$ or $\frac{q+1}{2}$. Considering all the possibilities, we get $(|\alpha(l)|,|\beta(l)|,|\gamma(l)|)=\left(\frac{q+1}{2},\frac{q+1}{2},0\right)$ or $(q,0,1)$. This proves (ii).
\end{proof}

Recall that an {\it arc} in the projective plane $PG(2,q)$ is a nonempty set of points such that no three of them are contained in the same line. A line of $PG(2,q)$ is called {\it external} or {\it secant} with respect to a given arc according as it meets the arc in $0$ or $2$ points. If $q$ is odd, then the maximum size of an arc is $q+1$. An {\it oval} is an arc of size $q+1$.

\begin{corollary}\label{cor-arc}
Let $\pi$ be a secant plane. Then the following hold.
\begin{enumerate}
\item [(a)] The set $\gamma(\pi)$ is an arc in $\pi$.
\item [(b)] The lines of $\pi$, which are secant with respect to the arc $\gamma(\pi)$, are contained in $\mathcal{S}_\pi$.
\end{enumerate}
\end{corollary}

\begin{proof}
(a) Note that the set $\gamma(\pi)$ is nonempty. This follows from the fact that $|\gamma(l)|\geq 1$ for any line $l\in\mathcal{S}_\pi$ (Lemma \ref{points-l}(i)). If there is a line $l$ of $\pi$ containing at least three points of $\gamma(\pi)$, then $|\gamma(l)|$ would be at least $3$, which is not possible by Lemma \ref{points-l}.

(b) If $l$ is a line of $\pi$ containing two points of $\gamma(\pi)$, then $|\gamma(l)|=2$ and hence $l$ must be a line of $\mathcal{S}_\pi$, which follows from Lemma \ref{points-l}.
\end{proof}

\begin{lemma}\label{recall}
Let $\pi$ be a secant plane. If $|\gamma(\pi)|=k$, then $|\alpha(\pi)|=\frac{k(k-1)}{2}$.
\end{lemma}

\begin{proof}
By Lemma \ref{points-l}(i), each line of $\mathcal{S}_\pi$ contains either $1$ or $2$ points of $\gamma(\pi)$. Note that $|\gamma(l)|=1$ for any line $l$ of $\mathcal{S}_\pi$ which contains a unique point of $\gamma(\pi)$. For such a line $l$, we have $|\alpha(l)|=0$ (follows from Lemma \ref{points-l}(i)) and hence $l$ does not contain any point of $\alpha(\pi)$. Similarly, $|\gamma(l)|=2$ for any line $l$ of $\mathcal{S}_\pi$ which contains two points of $\gamma(\pi)$ and in that case, $l$ contains $\frac{q-1}{2}$ points of $\alpha(\pi)$. Counting the cardinality of the set $Y=\{(x,l):x\in \alpha(\pi),\ l\in\mathcal{S}_\pi\ \mbox{and}\ x\in l\}$, we get
$$|\alpha(\pi)|\times \frac{q-1}{2}=|Y|=0+\frac{k(k-1)}{2}\times \frac{q-1}{2}.$$
This gives $|\alpha(\pi)|=\frac{k(k-1)}{2}$.
\end{proof}

\begin{lemma}\label{lem-oval}
For any secant plane $\pi$, the set $\gamma(\pi)$ is an oval in $\pi$ and so $|\gamma(\pi|=q+1$. Further, $\mathcal{S}_\pi$ is precisely the set of lines of $\pi$ which are secant with respect to $\gamma(\pi)$.
\end{lemma}

\begin{proof}
Let $|\gamma(\pi)|=k\geq 1$. In order to prove that $\gamma(\pi)$ is an oval in $\pi$, it is enough to show that $k=q+1$ by Corollary \ref{cor-arc}(a). We have $|\alpha(\pi)|=\frac{k(k-1)}{2}$ by Lemma \ref{recall}.

Fix a point $x\in \gamma(\pi)$ and let $l_1,l_2,\ldots,l_{q+1}$ be the $q+1$ lines of $\pi$ through $x$. Note that there are $q$ lines of $\mathcal{S}_\pi$ through $x$. Since $\gamma(\pi)$ is a $k$-arc in $\pi$, there are $k-1$ lines of $\mathcal{S}_\pi$, say $l_1,l_2,\ldots,l_{k-1}$, through $x$ each of which contains two points of $\gamma(\pi)$. There are $q-(k-1)$ lines of $\mathcal{S}_\pi$, say $l_k,l_{k+1},\ldots, l_{q}$, each of which contains a unique point (namely, $x$) of $\gamma(\pi)$. The line $l_{q+1}$ through $x$ is not a line of $\mathcal{S}_\pi$ and contains one point (namely, $x$) of $\gamma(\pi)$.
Since $l_i\in \mathcal{S}_\pi$ for $1\leq i\leq q$, we have $|\gamma(l_i)|=2$ for $1\leq i\leq k-1$ and $|\gamma(l_{i})|=1$ for $k\leq i\leq q$. Then Lemma \ref{points-l}(i) implies that each of the lines $l_1,l_2,\ldots,l_{k-1}$ contains $\frac{q-1}{2}$ points of $\alpha(\pi)$ and each of the lines $l_k,l_{k+1},\ldots,l_{q}$ contains no point $\alpha(\pi)$. Since $l_{q+1}\notin \mathcal{S}_\pi$ and $|\gamma(l_{q+1})|=1$, Lemma \ref{points-l}(ii) implies that the line $l_{q+1}$ contains $q$ points of $\alpha(\pi)$. Therefore, we get $|\alpha(\pi)|=(k-1)\times \frac{q-1}{2}+ (q-(k-1))\times 0+ q$. Thus, we have  
$$(k-1)\times \frac{q-1}{2}+ q=\frac{k(k-1)}{2}.$$\\
On solving the above equation, we get $k=-1 $ or $q+1$. Since $k\geq 1$, we must have $k=q+1$.

Since $\gamma(\pi)$ is an oval in $\pi$, the number of lines of $\pi$ which are secant to $\gamma(\pi)$ is equal to $\frac{q(q+1)}{2}=|\mathcal{S}_\pi|$. Therefore, by Corollary \ref{cor-arc}(b), $\mathcal{S}_\pi$ is precisely the set of secant lines to $\gamma(\pi)$. This completes the proof.
\end{proof}

\begin{corollary}\label{size-alpha-beta}
Let $\pi$ be a secant plane. Then $|\alpha(\pi)|=\frac{q^2+q}{2}$ and $|\beta (\pi)|=\frac{q^2-q}{2}$.
\end{corollary}

\begin{proof}
We have $|\gamma(\pi)|=q+1$ by Lemma \ref{lem-oval} and so $|\alpha(\pi)|=\frac{q^2+q}{2}$ by Lemma \ref{recall}. Since $|\alpha(\pi)|+|\beta(\pi)|+|\gamma(\pi)|=q^2+q+1$, it follows that $|\beta (\pi)|=\frac{q^2-q}{2}$.
\end{proof}

\section{Black points}

Recall that every point of $PG(3,q)$ is contained in $q^2$ or $\frac{q(q+1)}{2}$ lines of $\mathcal{S}$ by property (P1) and the black points are the ones which are contained in $q^2$ lines of $\mathcal{S}$.

\begin{lemma} \label{lem-black}
If $\pi$ is a secant plane, then the set of black points in $\pi$ is contained in the oval $\gamma(\pi)$.
\end{lemma}

\begin{proof}
Let $x$ be a black point in $\pi$. Suppose that $x$ is not contained in $\gamma(\pi)$. Fix a line $l$ of $\mathcal{S}_\pi$ through $x$ and consider the $q+1$ planes of $PG(3,q)$ through $l$. There are $q^2$ lines of $\mathcal{S}$ through $x$ and each of them is contained in some plane through $l$. Since $x\notin\gamma(\pi)$, the plane $\pi$ contains at most $\frac{q+1}{2}$ lines of $\mathcal{S}$ through $x$. Each of the remaining $q$ planes through $l$ contains at most $q$ lines of $\mathcal{S}$ through $x$. This implies that there are at most $\frac{q+1}{2}+q(q-1)$ lines of $\mathcal{S}$ through $x$. This is not possible, as $\frac{q+1}{2}+q(q-1)< q^2$. So $x\in\gamma(\pi)$.
\end{proof}

\begin{corollary} \label{coro-lem-black}
Let $\pi$ be a secant plane and $x$ be a black point of $\pi$. Then there are exactly $q$ lines of $\pi$ through $x$ which are contained in $\mathcal{S}$.
\end{corollary}

\begin{proof}
This follows from the fact that $x$ is contained in the oval $\gamma(\pi)$ by Lemma \ref{lem-black}.
\end{proof}

\begin{lemma}\label{lem-black-secant}
The number of black points in a given secant plane is independent of that plane.
\end{lemma}

\begin{proof}
Let $\pi$ be a secant plane and $\lambda_\pi$ denote the number of black points in $\pi$. We count the total number of lines of $\mathcal{S}$. The lines of $\mathcal{S}$ are divided into two types:
\begin{enumerate}
\item[(I)] the $\frac{q(q+1)}{2}$ lines of $\mathcal{S}$ which are contained in $\pi$,
\item[(II)] those lines of $\mathcal{S}$ which meet $\pi$ in a singleton.
\end{enumerate}
Let $\theta$ be the number of type (II) lines of $\mathcal{S}$. In order to calculate $\theta$, we divide the points of $\pi$ into four groups:
\begin{enumerate}
\item[(a)] The $\lambda_\pi$ black points contained in $\pi$: These points are contained in $\gamma(\pi)$ by Lemma \ref{lem-black}. Out of the $q^2$ lines of $\mathcal{S}$ through such a point, $q$ of them are contained in $\pi$.
\item[(b)] The $|\gamma(\pi)|-\lambda_\pi$ points of $\gamma(\pi)$ which are not black: Out of the $\frac{q(q+1)}{2}$ lines of $\mathcal{S}$ through such a point, $q$ of them are contained in $\pi$.
\item[(c)] The points of $\alpha(\pi)$: Out of the $\frac{q(q+1)}{2}$ lines of $\mathcal{S}$ through such a point, $\frac{q-1}{2}$ of them are contained in $\pi$.
\item[(d)] The points of $\beta(\pi)$: Out of the $\frac{q(q+1)}{2}$ lines of $\mathcal{S}$ through such a point, $\frac{q+1}{2}$ of them are contained in $\pi$.
\end{enumerate}
Using the values of $|\alpha(\pi)|$, $|\beta(\pi)|$ and $|\gamma(\pi)|$ obtained in Lemma \ref{lem-oval} and Corollary \ref{size-alpha-beta}, we get
\begin{align*}
\theta & = \lambda_\pi(q^2-q)+(q+1-\lambda_\pi)\left(\frac{q(q+1)}{2}-q\right)+|\alpha(\pi)|\left(\frac{q(q+1)}{2}-\frac{q-1}{2}\right)\\
 & \qquad\qquad +|\beta(\pi)|\left(\frac{q(q+1)}{2}-\frac{q+1}{2}\right) \\
 & = \lambda_\pi\left(\frac{q^2-q}{2}\right)+\frac{q^3(q+1)}{2}.
\end{align*}
Then $|\mathcal{S}| = \theta +\dfrac{q(q+1)}{2}= \lambda_\pi\left(\dfrac{q^2-q}{2}\right)+ \dfrac{q^4+q^3+q^2+q}{2}$.
Since $|\mathcal{S}|$ is a fixed number, it follows that $\lambda_\pi$ is independent of the secant plane $\pi$.
\end{proof}

By Lemma \ref{lem-black-secant}, we denote by $\lambda$ the number of black points in a secant plane. From the proof of Lemma \ref{lem-black-secant}, we thus have the following equation involving $\lambda$ and $|\mathcal{S}|$:
\begin{equation}\label{eq-secant-black}
\lambda\left(\frac{q^2-q}{2}\right)+\frac{q^4+q^3+q^2+q}{2}=|\mathcal{S}|.
\end{equation}
As a consequence of Lemma \ref{lem-black}, we have
\begin{corollary}\label{coro-black}
$\lambda\leq q+1$.
\end{corollary}

\begin{lemma} \label{lem-black-tangent}
The number of black points in a given tangent plane is independent of that plane.
\end{lemma}

\begin{proof}
Let $\pi$ be a tangent plane with pole $p_\pi$ and $\mu_\pi$ be the number of black points in $\pi$. We shall apply a similar argument as in the proof of Lemma \ref{lem-black-secant} by calculating $|\mathcal{S}|$.
The lines of $\mathcal{S}$ are divided into two types: (I) the $q^2$ lines of $\mathcal{S}$ which are contained in $\pi$, and (II) those lines of $\mathcal{S}$ which meet $\pi$ in a singleton.
Let $\theta$ be the number of type (II) lines of $\mathcal{S}$. In order to calculate $\theta$, we divide the points of $\pi$ into two groups:
\begin{enumerate}
\item[(a)] The $\mu_\pi$ black points contained in $\pi$,
\item[(b)] The $q^2+q+1-\mu_\pi$ points of $\pi$ which are not black.
\end{enumerate}
If $x$ is a point of $\pi$ which is different from $p_\pi$, then Lemma \ref{size-a_pi} implies that the number of lines of $\mathcal{S}$ through $x$ which are not contained in $\pi$ is $q^2 -q$ or $\frac{q(q+1)}{2}-q$ according as $x$ is a black point or not. We consider two cases depending on $p_\pi$ is a black point or not.\medskip

\noindent {\bf Case-1:} $p_\pi$ is a black point. In this case, Lemma \ref{size-a_pi} implies that none of the $q^2$ lines of $\mathcal{S}$ through $p_\pi$ is contained in $\pi$. Then
\begin{align*}
\theta & = q^2+(\mu_\pi-1)(q^2-q)+(q^2+q+1-\mu_\pi)\left(\frac{q(q+1)}{2}-q\right)\\
  & = \mu_\pi\left(\frac{q^2-q}{2}\right)+\frac{q^4+q}{2}.
\end{align*}

\noindent {\bf Case-2:} $p_\pi$ is not a black point. In this case, none of the $\frac{q(q+1)}{2}$ lines of $\mathcal{S}$ through $p_\pi$ is contained in $\pi$ by Lemma \ref{size-a_pi}. Then
\begin{align*}
\theta & = \mu_\pi(q^2-q)+ \frac{q(q+1)}{2}+(q^2+q-\mu_\pi)\left(\frac{q(q+1)}{2}-q\right)\\
  & = \mu_\pi\left(\frac{q^2-q}{2}\right)+\frac{q^4+q}{2}.
\end{align*}
In both cases, $|\mathcal{S}| = \theta +q^2= \mu_\pi\left(\dfrac{q^2-q}{2}\right)+ \dfrac{q^4+2q^2+q}{2}$.
Since $|\mathcal{S}|$ is a fixed number, it follows that $\mu_\pi$ is independent of the tangent plane $\pi$.
\end{proof}

By Lemma \ref{lem-black-tangent}, we denote by $\mu$ the number of black points in a tangent plane. From the proof of Lemma \ref{lem-black-tangent}, we thus have the following equation involving $\mu$ and $|\mathcal{S}|$:
\begin{equation}\label{eq-tangent-black}
\mu\left(\frac{q^2-q}{2}\right)+\frac{q^4+2q^2+q}{2}=|\mathcal{S}|
\end{equation}
From equations (\ref{eq-secant-black}) and (\ref{eq-tangent-black}), we have
\begin{equation}\label{lambda-mu}
\mu=\lambda+q.
\end{equation}

\begin{lemma}\label{0-1-2-q+1}
The following hold:
\begin{enumerate}
\item[(i)] Every line of $PG(3,q)$ contains $0,1,2$ or $q+1$ black points.
\item[(ii)] If a line of $PG(3,q)$ contains exactly two black points, then it is a line of $\mathcal{S}$.
\end{enumerate}
\end{lemma}

\begin{proof}
Let $l$ be a line of $PG(3,q)$ and $b$ be the number of black points contained in $l$. Assume that $b>2$. If there exists a secant plane $\pi$ through $l$, then Lemma \ref{lem-black} implies that the line $l$ contains $b\geq 3$ number of points of the oval $\gamma(\pi)$ in $\pi$, which is not possible. So all planes through $l$ are tangent planes. Then all the $q+1$ points of $l$ are black by Lemma \ref{lem-plane-point}. This proves (i).

If $b=2$, then Lemma \ref{lem-plane-point} implies that there exists a secant plane $\pi$ through $l$. By Lemma \ref{lem-black}, $l$ is a secant line of $\pi$ with respect to the oval $\gamma(\pi)$. So $l$ is a line of $\mathcal{S}_{\pi}$ by the second part of Lemma \ref{lem-oval} and hence $l$ is a line of $\mathcal{S}$. This proves (ii).
\end{proof}

\section{Proof of Theorem \ref{q-odd}}

We shall continue with the notation used in the previous sections. We denote by $\mathcal{H}$ the set of all black points of $PG(3,q)$, and by $\mathcal{H}_\pi$ the set of black points of $PG(3,q)$ which are contained in a given plane $\pi$.

\begin{lemma} \label{lem-size-H}
$|\mathcal{H}|=\lambda(q+1)$. In particular, $|\mathcal{H}|\leq (q+1)^2$.
\end{lemma}

\begin{proof}
Fix a secant plane $\pi$. Let $l$ be a line of $\pi$ which is external to the oval $\gamma(\pi)$. By Lemma \ref{lem-black}, none of the points of $l$ is black. Then, by Lemma \ref{lem-plane-point}, each plane through $l$ is a secant plane. The number of black points contained in a secant plane is $\lambda$. Counting all the black points contained in the $q+1$ planes through $l$, we get $|\mathcal{H}|=\lambda(q+1)$. Since $\lambda\leq q+1$ by Corollary \ref{coro-black}, we have $|\mathcal{H}|\leq (q+1)^2$.
\end{proof}

The following result was proved by Bose and Burton in \cite[Theorem 1]{Bos-But}. We need it in the plane case.
\begin{proposition} \cite{Bos-But}\label{BB}
Let $B$ be a set of points of $PG(n,q)$ such that every line of $PG(n,q)$ meets $B$. Then $|B|\geq (q^n-1)/(q-1)$, and equality holds if and only if $B$ is a hyperplane of $PG(n,q)$.
\end{proposition}

\begin{lemma}\label{exis-line}
If $\pi$ be a tangent plane, then $\mathcal{H}_\pi$ contains a line.
\end{lemma}

\begin{proof}
By Corollary \ref{cor-plane}, every line of $\pi$ meets $\mathcal{H}_\pi$. By Proposition \ref{BB} (taking $n=2$), we have $|\mathcal{H}_\pi|\geq q+1$, and equality holds if and only if $\mathcal{H}_\pi$ itself is a line of $\pi$.

Therefore, assume that $|\mathcal{H}_\pi|> q+1$. Since $q$ is odd, the maximum size of an arc in $\pi$ is $q+1$. So $\mathcal{H}_\pi$ cannot be an arc and hence there exists a line $l$ of $\pi$ which contains at least three points of $\mathcal{H}_\pi$. Then all points of $l$ are black by Lemma \ref{0-1-2-q+1}(i) and so $l$ is contained in $ \mathcal{H}_\pi$.
\end{proof}

\begin{lemma}\label{to-use}
Let $\pi$ be a tangent plane. Then $\mathcal{H}_\pi$ is either a line or union of two (intersecting) lines.
\end{lemma}

\begin{proof}
Since $q\geq 3$, using Lemmas \ref{0-1-2-q+1}(i) and \ref{exis-line}, observe that there are only four possibilities for $\mathcal{H}_\pi$:
\begin{enumerate}
\item[(1)] $\mathcal{H}_\pi$ is a line.
\item[(2)] $\mathcal{H}_\pi$ is the union of a line $l$ and a point of $\pi$ not contained in $l$.
\item[(3)] $\mathcal{H}_\pi$ is the union of two (intersecting) lines.
\item[(4)] $\mathcal{H}_\pi$ is the whole plane $\pi$.
\end{enumerate}

We show that the possibilities (2) and (4) do not occur. If $\mathcal{H}_\pi$ is the whole plane $\pi$, then $\mu=q^2+q+1$ and so $\lambda=q^2+1$ by equation (\ref{lambda-mu}), which is not possible by Corollary \ref{coro-black}.

Now suppose that $\mathcal{H}_\pi$ is the union of a line $l$ and a point $x$ not on $l$. If $p_\pi\neq x$, then take $t$ to be the line through $p_\pi$ and $x$ (note that $p_\pi$ may or may not be on $l$). If $p_\pi=x$, then take $t$ to be any line through $p_\pi=x$. Since $\pi$ is a tangent plane, $t$ is not a line of $\mathcal{S}_\pi$ by Corollary \ref{cor-pole} and hence is not a line of $\mathcal{S}$. On the other hand, since $t$ contains only two black points (namely, the point $x$ and the intersection point of $l$ and $t$), $t$ is a line of $\mathcal{S}$ by Lemma \ref{0-1-2-q+1}(ii). This leads to a contradiction.
\end{proof}

\begin{lemma}
Let $\pi$ be a tangent plane. If $\mathcal{H}_\pi$ is a line of $PG(3,q)$, then the following hold:
\begin{enumerate}
\item[(i)] $\mathcal{H}_\pi$ is not a line of $\mathcal{S}$.
\item[(ii)] $\mathcal{H}_\pi=\mathcal{H}$.
\item[(iii)] $\mathcal{S}$ is a set of $\dfrac{q^4+q^3+2q^2}{2}$ lines of $PG(3,q)$ not containing the line $\mathcal{H}$.
\end{enumerate}
\end{lemma}

\begin{proof}
(i) Suppose that $\mathcal{H}_\pi$ is a line of $\mathcal{S}$. Then, by Corollary \ref{cor-pole}, the pole $p_\pi$ of $\pi$ must be a point of $\pi \setminus \mathcal{H}_\pi$. Fix a line $m$ of $\pi$ through $p_\pi$. Note that $m\notin \mathcal{S}$ again by Corollary \ref{cor-pole}. Let $x$ be the point of intersection of $m$ and $\mathcal{H}_\pi$. Since $m$ contains only one black point (which is $x$), Lemma \ref{lem-plane-point} implies that $m$ is contained in one tangent plane (namely, $\pi$) and $q$ secant planes. Since $x\neq p_\pi$, by Lemma \ref{size-a_pi}, there are $q$ lines of $\pi$ through $x$ which are contained in $\mathcal{S}$. In each of the $q$ secant planes through $m$, by Corollary \ref{coro-lem-black}, there are $q$ lines through $x$ which are contained in $\mathcal{S}$. Since $m\notin \mathcal{S}$, we get $q(q+1)=q^2+q$ lines of $\mathcal{S}$ through $x$, which is not possible by property (P1).\medskip

(ii) Suppose that $x$ is a black point which is not contained in $\mathcal{H}_\pi$. Let $\pi'$ be the plane generated by the line $\mathcal{H}_\pi$ and the point $x$. We have $\pi\neq \pi'$ as $x$ is not a black point of $\pi$. Each of the planes through the line $\mathcal{H}_\pi$ is a tangent plane by Lemma \ref{lem-plane-point}. In particular, $\pi'$ is a tangent plane. Note that $\pi$ contains $q+1$ black points, whereas $\pi'$ contains at least $q+2$ black points. This contradicts Lemma \ref{lem-black-tangent}.\medskip

(iii) The line $\mathcal{H}$ is not contained in $\mathcal{S}$ by (i) and (ii). Since the tangent plane $\pi$ contains $q+1$ black points, we have $\mu=q+1$. Then equation (\ref{eq-tangent-black}) gives that $|\mathcal{S}|= \dfrac{q^4+q^3+2q^2}{2}$.
\end{proof}

In the rest of this section, we assume that $\mathcal{H}_\pi$ is the union of two (intersecting) lines for every tangent plane $\pi$. So $\mu=2q+1$ and then equation (\ref{lambda-mu}) gives that $\lambda=q+1$. From equation (\ref{eq-tangent-black}) and Lemma \ref{lem-size-H}, we get

\begin{equation}\label{size-S-H}
|\mathcal{S}|=\frac{q^2(q+1)^2}{2} \;\;\;\; \ \mbox{and}  \;\;\;\; |\mathcal{H}|=(q+1)^2.
\end{equation}

\begin{lemma}\label{pole-inters}
Let $\pi$ be a tangent plane. If $\mathcal{H}_\pi$ is the union of the lines $l$ and $l'$ of $\pi$, then the pole $p_\pi$ of $\pi$ is the intersection point of $l$ and $l^\prime$.
\end{lemma}

\begin{proof}
Let $x$ be the intersection point of $l$ and $l^\prime$. Suppose that $p_\pi\neq x$. Let $t$ be a line of $\pi$ through $p_\pi$ which does not contain $x$ (note that $p_\pi$ may or may not be contained in $l\cup l'$). Since $\pi$ is a tangent plane, $t$ is not a line of $\mathcal{S}_\pi$ by Corollary \ref{cor-pole} and hence is not a line of $\mathcal{S}$. On the other hand, since $t$ contains two black points (namely, the two intersection points of $t$ with $l$ and $l'$), it is a line of $\mathcal{S}$ by Lemma \ref{0-1-2-q+1}(ii). This leads to a contradiction.
\end{proof}

We call a line of $PG(3,q)$ {\it black} if it is contained in $\mathcal{H}$.

\begin{lemma}\label{black-line-3}
Every black point is contained in at most two black lines.
\end{lemma}

\begin{proof}
Let $x$ be a black point. If possible, suppose that there are three distinct black lines $l,l_1,l_2$ each of which contains $x$. Let $\pi$ (respectively, $\pi'$) be the plane generated by $l,l_1$ (respectively, $l,l_2$). Each plane through $l$ is a tangent plane by Lemma \ref{lem-plane-point}. So $\pi$ and $\pi'$ are tangent planes. Since $\mathcal{H}_{\pi}=l\cup l_1$ and $\mathcal{H}_{\pi'}=l\cup l_2$, it follows that $\pi\neq \pi'$. By Lemma \ref{pole-inters}, $x$ is the pole of both $\pi$ and $\pi'$. So the lines through $x$ which are contained in $\pi$ or $\pi'$ are not lines of $\mathcal{S}$ by Corollary \ref{cor-pole}. Thus each line of $\mathcal{S}$ through $x$ is contained in some plane through $l$ which is different from both $\pi$ and $\pi'$. It follows that the number of lines of $\mathcal{S}$ through $x$ is at most $q(q-1)$, which contradicts to the fact that there are $q^2$ lines of $\mathcal{S}$ through $x$ (being a black point).
\end{proof}

\begin{lemma}\label{black-line-2}
Every black point is contained in precisely two black lines.
\end{lemma}

\begin{proof}
Let $x$ be a black point and $l$ be a black line containing $x$. The existence of such a line $l$ follows from the facts that $x$ is contained in a tangent plane (Corollary \ref {cor-plane-1}) and that the set of all black points in that tangent plane is a union of two black lines. By Lemma \ref{lem-plane-point}, let $\pi_1,\pi_2,\ldots,\pi_{q+1}$ be the $q+1$ tangent planes through $l$. For $1\leq i\leq q+1$, we have $\mathcal{H}_{\pi_i}=l\cup l_i$ for some black line $l_i$ of $\pi_i$ different from $l$. Let $\{p_i\}=l\cap l_i$. Lemma \ref{black-line-3} implies that $p_i\neq p_j$ for $1\leq i\neq j\leq q+1$, and so $l=\{p_1,p_2,\ldots, p_{q+1}\}$. Since $x\in l$, we have $x=p_j$ for some $1\leq j\leq q+1$. Thus, applying Lemma \ref{black-line-3} again, it follows that $x$ is contained in precisely two black lines, namely, $l$ and $l_j$.
\end{proof}

We refer to \cite{PT} for the basics on finite generalized quadrangles. Let $s$ and $t$ be positive integers. A {\it generalized quadrangle} of order $(s,t)$ is a point-line geometry $\mathcal{X}=(P,L)$ with point set $P$ and line set $L$ satisfying the following three axioms:
\begin{enumerate}
\item[(Q1)] Every line contains $s+1$ points and every point is contained in $t+1$ lines.
\item[(Q2)] Two distinct lines have at most one point in common (equivalently, two distinct points are contained in at most one line).
\item[(Q3)] For every point-line pair $(x,l)\in P\times L$ with $x\notin l$, there exists a unique line $m\in L$ containing $x$ and intersecting $l$.
\end{enumerate}

Let $\mathcal{X}=(P,L)$ be a generalized quadrangle of order $(s,t)$. Then, $|P|=(s+1)(st+1)$ and $|L|=(t+1)(st+1)$ \cite[1.2.1]{PT}. If $P$ is a subset of the point set of some projective space $PG(n,q)$, $L$ is a set of lines of $PG(n,q)$ and $P$ is the union of all lines in $L$, then $\mathcal{X}=(P,L)$ is called a {\it projective generalized quadrangle}. The points and the lines contained in a hyperbolic quadric in $PG(3,q)$ form a projective generalized quadrangle of order $(q,1)$. Conversely, any projective generalized quadrangle of order $(q,1)$ with ambient space $PG(3,q)$ is a hyperbolic quadric in $PG(3,q)$, this follows from \cite[4.4.8]{PT}.\\

The following two lemmas complete the proof of Theorem \ref{q-odd}.

\begin{lemma}\label{hyperbolic}
The points of $\mathcal{H}$ together with the black lines form a hyperbolic quadric in $PG(3,q)$.
\end{lemma}

\begin{proof}
We have $|\mathcal{H}|=(q+1)^2$ by (\ref{size-S-H}). It is enough to show that the points of $\mathcal{H}$ together with the black lines form a projective generalized quadrangle of order $(q,1)$.

Each black line contains $q+1$ points of $\mathcal{H}$. By Lemma \ref{black-line-2}, each point of $\mathcal{H}$ is contained in exactly two black lines. Thus the axiom (Q1) is satisfied with $s=q$ and $t=1$. Clearly, the axiom (Q2) is satisfied.

We verify the axiom (Q3). Let $l=\{x_1,x_2,\ldots,x_{q+1}\}$ be a black line and $x$ be a black point not contained in $l$. By Lemma \ref{black-line-2}, let $l_i$ be the second black line through $x_i$ (different from $l$) for $1\leq i\leq q+1$. If $l_i$ and $l_j$ intersect for $i\neq j$, then the tangent plane $\pi$ generated by $l_i$ and $l_j$ contains $l$ as well. This implies that $\mathcal{H}_{\pi}$ contains the union of three distinct black lines (namely, $l,l_i,l_j$), which is not possible. Thus the black lines $l_1,l_2,\ldots,l_{q+1}$ are pairwise disjoint. These $q+1$ black lines contain $(q+1)^2$ black points and hence their union must be equal to $\mathcal{H}$. In particular, $x$ is a point of $l_j$ for unique $j\in\{1,2,\ldots,q+1\}$. Then $l_j$ is the unique black line containing $x_j$ and intersecting $l$.

From the above two paragraphs, it follows that the points of $\mathcal{H}$ together with the black lines form a projective generalized quadrangle of order $(q,1)$. This completes the proof.
\end{proof}

\begin{lemma}
The lines of $\mathcal{S}$ are precisely the secant lines to the hyperbolic quadric $\mathcal{H}$.
\end{lemma}

\begin{proof}
By (\ref{size-S-H}), we have $|\mathcal{S}|=\dfrac{q^2(q+1)^2}{2}$, which is equal to the number of secant lines to $\mathcal{H}$. It is enough to show that every secant line to $\mathcal{H}$  is a line of $\mathcal{S}$. This follows from Lemma \ref{0-1-2-q+1}(ii), as every secant line to $\mathcal{H}$ contains exactly two black points.
\end{proof}
\section*{Acknoledgments}
We thank Dr. Binod Kumar Sahoo for his helpful comments. The first author was supported by the Council of Scientific and Industrial Research grant No. 09/1002(0040)/2017-EMR-I, Ministry of Human Resource Development, Government of India and the second author was supported by Department of Science and Technology grant No. EMR/2016/006624, by the UGC Center for Advanced Studies and by National Board for Higher Mathematics grant No. 0204/18/2019/R\&D-II/10462, Department of Atomic Energy, Government of India.

\vskip .5cm

\noindent {\bf Addresses}:\\

\noindent {\bf Puspendu Pradhan} (Email: puspendu.pradhan@niser.ac.in)
\begin{enumerate}
\item[1)] School of Mathematical Sciences, National Institute of Science Education and Research (NISER), Bhubaneswar, P.O.- Jatni, District- Khurda, Odisha - 752050, India.

\item[2)] Homi Bhabha National Institute (HBNI), Training School Complex, Anushakti Nagar, Mumbai - 400094, India.\\
\end{enumerate}

\noindent {\bf Bikramaditya Sahu} (Email: bikramadityas@iisc.ac.in)
\begin{enumerate}
\item[1)] Department of Mathematics, Indian Institute of Science, Bangalore - 560012, India.\\
\end{enumerate}

\end{document}